%% file: standard_modules_v7.tex
\documentclass[10pt]{article}

\usepackage{amsfonts}
\usepackage{amsmath}
\usepackage{amssymb}
\usepackage{amsthm}
\usepackage[alphabetic]{amsrefs}
\usepackage{calc}
\usepackage{graphicx}
\usepackage[hidelinks]{hyperref}
\usepackage[left=1in,top=1in,right=1in,foot=1in]{geometry}
\usepackage{todonotes}
\usepackage{microtype}
\usepackage{tikz-cd}

\input{preamble.tex}

\newcommand{\Ch}{\mathrm{ch}}
\newcommand{\Td}{\mathrm{Td}}
\newcommand{\Forg}{\mathrm{Forg}}

\providecommand{\keywords}[1]{\textbf{\textit{Keywords---}} #1}

\title{Standard modules of affine Hecke algebras}
\date{\today}
\author{Stefan Dawydiak \thanks{School of Mathematics and Statistics, University of Glasgow; email \texttt{stefan.dawydiak@glasgow.ac.uk}}}

\begin{document}
\maketitle
\begin{abstract}
Let $G$ be a connected reductive group defined and split over a non-archimedean local field $F$. We give a new geometric proof of a special case of a recent theorem of Solleveld. Namely, we show that the class of standard Iwahori-spherical $G(F)$-representations, a notion a priori dependent on the coefficient field being the complex numbers, is actually defined over $\bar{\Q}_\ell$. 

An unpublished theorem of Clozel, proven with global techniques, says that the class of essentially square-integrable representations is also defined over $\bar{\Q}_\ell$. As an application of our main result, we give a local proof of this theorem for inner forms of $\GL_n$, as well as showing that standard representations of these groups are defined over $\bar{\Q}_\ell$.
\end{abstract}
%
%\tableofcontents

%\todo[inline]{Make v3, with remark about future using Kato, etc, for unipotent blocks of classical groups. Aim for Friday or weekend, look at posting schedules.}

%\todo[inline]{Use Kato for type BC. Need:
%%
%
%
%Affine paving of his fixed pints of Spr. Fibres. This should be his Chern character surjectivity statement. Can do inductively with support filtration: surjection should induce surjection of associated gradeds, so $H_0$ therefore spanned by $0$-cycles ie points, then $H_1$ is spanned by classes of supports of coherent sheaves on one-dimensional subvarities, modulo points--which we already have---, and so on.
%
%Rational points in his ``nilpotent orbits" Might have to be creative to apply Kottwitz, or it will be obvious because classical groups.
%%
%%
%Casselman criterion. In Kato.
%
%Induction theorem to convert tempered and discrete to standards
%}

%\todo[inline]{change Baum et. al. references to Fulton, cleaner.}
\section{Introduction}
\subsection{Representations over $\bar{\Q}_\ell$}

Let $G$ be a connected reductive group defined and split over a non-archimedean local field $F$. Let $\pi$ be a complex admissible representation of $G(F)$, let $\gamma$ be a field automorphism of $\C$, and define a representation $\gamma\cdot\pi$ of $G(F)$ by declaring that its matrix coefficients are of the form $\gamma(f(g))$, for matrix coefficients $f$ of $\pi$ (see Section \ref{subsection action of field automorphisms} for a precise definition). It is clear that $\gamma\cdot\pi$ is irreducible when $\pi$ is, and likewise for supercuspidality. In \cite{KSV} and \cite{SolleveldStandard}, less obvious heritability from $\pi$ to $\gamma\cdot\pi$ is studied. 
Recall that the standard representations of $G(F)$ are by definition parabolic inductions $i_P^G(\sigma\otimes\nu)$ for irreducible tempered representations $\sigma$ and positive unramified characters $\nu$ of $M_P(F)$, respectively, where $M_P$ is the Levi subgroup of $P$ (see \cite[Def. 4.1, Thm. 4.3 (a)]{SolleveldStandard}). Thus the standard modules occur in algebraic families.

As pointed out in \cite{SolleveldStandard}, this occurrence in families is one reason to think that the standard representations may be better behaved in the context of the categorical Local Langlands correspondence of \cite{FS} than simple representations. For this to make sense, though, the notion of ``standard representation" should be defined over $\bar{\Q}_\ell$, \textit{i.e.}, be stable under the above action of $\Aut(\C)$. To this end, Solleveld has recently shown
\begin{theorem}[\cite{SolleveldStandard}, Prop. 5.5]
\label{thm standard}
Let $G$ be a connected reductive group over $F$ and $\pi$ a standard representation of $G(F)$. Then $\gamma\cdot\pi$ is also a standard representation of $G(F)$.
\end{theorem}
The first result of this note is a new geometric proof, for representations in the principal block of split groups, of Theorem \ref{thm standard}.
Precisely, we give a geometric proof, as our Corollary \ref{Cor Hecke twist standard and formal degrees}, that if $K(s,N,\rho)$ is a standard module for the Iwahori-Hecke algebra $H$ of $G$, as constructed in terms of $G^\vee$ and the enhanced $L$-parameter $(s,N,\rho)$ in 
\cite{KLDeligneLanglands}, \cite{CG}, then $\gamma\cdot K(s,N,\rho)\simeq K(\gamma(s),N,\rho_\gamma)$, as one would expect upon considering central characters. Here, $\rho_\gamma$ is the pullback of $\rho$ along the isomorphism $\pi_0(Z_{G^\vee}(s,N))\simeq \pi_0(Z_{G^\vee}(\gamma(s),N))$ induced by $\gamma$, see \eqref{component group isomorphism} below. This implies that the Deligne-Langlands correspondence of \cite{KLDeligneLanglands} is defined over $\bar{\Q}_\ell$, see Corollary \ref{Cor DL}.

Note that if $\pi$ is generated by its Iwahori-fixed vectors for a chosen Iwahori subgroup $I$ of $G(F)$, then so is $\gamma\cdot\pi$. In general, the functor $\C_\gamma\otimes_\C-$ does not preserve Bernstein blocks (consider $G=\GL_1$), but still interacts well enough with types and Hecke algebras for our purposes, see Lemma \ref{lem types}. 

The main input to our argument is the main result of \cite{DLP}, which with \cite{KLDeligneLanglands}, \cite{CG} implies that, up the central character, representations $K(s,N,\rho)$ can be constructed on rationalized Chow groups of certain varieties defined over $\Q$. This may be viewed as further suggestion that  (at least for the principal block or $G=\GL_n$) the standard modules admit a natural interpretation in terms of \cite{FS}. 

After to appealing to \cite{DLP}, the argument is very formal, amounting essentially to various intersection-theoretic constructions commuting with isomorphisms of schemes, \textit{i.e.}, being well-defined. Nevertheless, we have tried to include careful references for these constructions beyond the usual references for singular homology, because it is conceptually important that we may work morphisms of schemes that are not morphisms of $\C$-schemes.
%Therefore pullback of cycles gives a natural linear map 
%%
%\[
%\C_\gamma\otimes_\C K(s,N,\rho)\to K(\gamma(s),N,\rho_\gamma).
%\]
%%

Recall that a Langlands parameter 
\[
W_F\times\SL_2(\C)\to G^\vee(\C)
\]
is called essentially discrete if it does not factor through any proper Levi subgroup of $G^\vee$. Clearly, this property is stable under post-composition with field automorphisms of $\C$. The corresponding property for representations is of course not obvious. Our Corollary \ref{Cor Hecke twist standard and formal degrees} also shows that when $K(s,N,\rho)$ is essentially square-integrable, then so is $\gamma\cdot K(s,N,\rho)$, with the same formal degree. 

Combining Corollary \ref{Cor Hecke twist standard and formal degrees} with the results of \cite{BushKutzBook}, \cite{BushKutzComp}, \cite{Secherre}, \cite{SecherreStevens}, \cite{SecherreStevensIV}, we obtain the second result of this note: We give a local proof, valid for function fields, for inner forms $G$ of $\GL_n$, of
\begin{theorem}[Clozel, unpublished]
\label{thm Clozel}
Let $G$ be a connected reductive group over a $p$-adic field $F$, and let $\pi$ be an essentially square-integrable representation of $G$. Then $\gamma\cdot\pi$ is again essentially square-integrable, with the same formal degree as $\pi$.
\end{theorem}
In fact, using \cite{SolleveldCompletion}, we treat general standard modules, giving a proof of Theorem \ref{thm standard} for inner forms of $\GL_n$, compatible with \ref{thm Clozel}.

An exposition of Clozel's proof is given in \cite{KSV}, where the corresponding result for function fields is then deduced from Theorem \ref{thm Clozel} by Kazhdan's method of close fields.

Finally, note that irreducible tempered and discrete series representations are standard, but that within the class of standard representations, the tempered representations are not stable under $\Aut(\C)$. Thus the standard representations are a class of representations defined over $\bar{\Q}_\ell$ containing the tempered representations. Our proof of Theorem \ref{thm standard} will show that the standard representations form essentially the minimal  enlargement of the tempered representations, at least in the case of the principal block and inner forms of $\GL_n$.

In \cite{BravermanKazhdan}, Braverman-Kazhdan introduce a certain subalgebra 
$\mathcal{J}$ of the Harish-Chandra Schwartz algebra of $G(F)$. By \cite{suzuki}, Theorem \ref{thm standard} implies that the class of simple $\mathcal{J}$-modules is defined over $\bar{\Q}_\ell$. Our proof in this note provides an independent argument for the principal block, and for inner forms of $\GL_n$. The idempotented piece $\mathcal{J}_I$ of $\mathcal{J}$ corresponding to the principal block is isomorphic to Lusztig's asymptotic Hecke algebra $J$ \cite{BravermanKazhdan} (see the proofs in \cite{Plancherel}, \cite{rigid}, \cite{BKK}), and \cite{Plancherel} then implies that $\mathcal{J}_I$ has a $\C$-basis of $\Q$-valued functions.

\subsection{Prior work for discrete series of inner forms of $\GL_n$}
Theorem \ref{thm Clozel} can be deduced from a local Langlands correspondence for $G$ satisfying the properties of \cite[Conj. 6.1]{KalethaTaibi}, see Proposition 6.3 of \textit{op. cit.}. Of course, this is an extreme approach: all the known proofs \cite{LRS}, \cite{Siyan} (for function fields), \cite{Henniart}, \cite{HT}, \cite{Scholze} (for $p$-adic fields) of the local Langlands correspondence for $\GL_n$ and its inner forms use far more global input than does Clozel's proof.

However, there are also prior local proofs of Theorem \ref{thm Clozel} for $\GL_n$ and its inner forms. After finishing a first version of this paper, we happened upon \cite[\S 2.1.1 (ii)]{Dat}, where it is explained that for $\GL_n(F)$, Theorem \ref{thm Clozel} is known to follow the Bernstein-Zelevinski classification, specifically \cite[Thm. 9.3]{ZII}, or alternatively from \cite{Rogawski} for $p$-adic fields (which again uses global techniques). Presumably the case of $\GL_n(D)$ then follows from the (purely local) proof of a Bernstein-Zelevinski-style classification for these groups given by M\'{i}nguez-S\'{e}cherre in \cite{MSBanal}. As pointed out in 
\cite[Rem. 10.1]{MSJLmodulol}, $\Aut(\C)$-equivariance of the Jacquet-Langlands correspondence follows from its character-theoretic characterization.

As \cite{BushKutzBook}, \cite{BushKutzComp} and \cite{Secherre}, \cite{SecherreStevens} are independent of \cite{ZII} and \cite{MSBanal}, respectively, our proof is independent of the above-recalled local proofs.
The question of standard representations is not taken up in the above works.

\subsection{Other groups}
In principle, our approach to Theorem \ref{thm Clozel} based on affine Hecke algebras from types should adapt more readily to general groups than approaches based on the Bernstein-Zelevinski classification. However, the techniques of the present note work only for equal-parameter affine Hecke algebras. While such algebras can appear for types beyond the principal block or inner forms of $\GL_n$, it seems that the only groups that can be dealt with via types using equal-parameter algebras exclusively are again inner forms of $\GL_n$. Using more general Morita equivalences, Aubert-Baum-Plymen-Solleveld showed that all Bernstein blocks of inner forms of $\SL_n$ may be treated by more complicated algebras built out of equal-parameter affine Hecke algebras \cite{ABPS1}, but Theorem \ref{thm SolleveldCompletion} connecting preservation of standardness to our geometric statement does not seem to be available for these algebras, see the discussion in \cite[p.43]{SolleveldCompletion}.

In future work, we hope to expand our techniques to unipotent blocks of classical groups, using \cite{Kato} and \cite{CiubotaruKato}.

%\todo[inline]{what about $G_2$? or $\SL_2$ Aubert-Xu or Suzuki-Xu Sp(4)?}

\subsection{Acknowledgements}
The author thanks Anne-Marie Aubert, Dinakar Muthiah, Evgeny Shinder, Maarten Solleveld, Eric Sommers, and Yakov Varshavsky for helpful conversations, and Maximilien Mackie for pointing out an imprecision in Remark \ref{rem EM} in an earlier version of this note.
Some of this work was completed during ``Geometric representation theory in the Langlands program" at CIRM Luminy; the author thanks the organizers and staff for their hospitality. 
The author was supported by Engineering and Physical Sciences Research Council grant UKRI167 ``Geometry of double loop groups." 

\section{Standard modules}
\label{section standard modules}
\subsection{Notation and conventions}
Let $F$ be a local non-archimedean field and let $G^\vee$ be a connected reductive group defined and split over $F$ (we adopt this convention on which group is dual, as the $p$-adic group will appear little in the sequel.) Let $q$ be the cardinality of the residue field of $F$.

Let $G$ be the group dual to $G^\vee$, considered as a group over $\C$, with root datum $(X^*,\Phi,X_*,\Phi^\vee)$, and chosen set $\Sigma$ of simple roots. We write $\Bb$ for the flag variety of $G$, $\Nn=T^*\Bb$, and $\St=\Nn\times_\mathcal{N}\Nn$ for the Steinberg variety of $G$, where $\mathcal{N}\subset\g=\Lie(G)$ is the nilpotent cone. All of these depend only on $G_{\mathrm{ad}}=G/Z(G)$, a semisimple group of adjoint type. Therefore they are each defined over $\Q$, by Chevalley's theorem (in fact, over $\Z$ \cite{LusJAMS2009}). Given a nilpotent element $N\in\g$, we write $\Bb_N$ for the Springer fibre, and given $g\in G$, we write $X^g$ for the fixed points for any $G$-scheme $X$. We will view the above $G$-schemes as $G\times\Gm$-schemes, such that $\Gm$ scales the cotangent fibres in $\Nn$.

We write $K(X)=K_0(\Coh(X))$ for the Grothendieck group, for any stack $X$.

Let $W$ be the Weyl group of $G$.
Let $H$ be the affine Hecke algebra corresponding to $\tilde{W}=W\ltimes X^*$ (recall our conventions on duals) over the formal Laurent polynomial ring 
$\Z[\bq^{1/2},\bq^{-1/2}]$. We view $\Z[\bq^{1/2},\bq^{-1/2}]=K_0(\pt/\Gm)$ and identify 
\[
H\simeq K(\St/G\times\Gm)
\]
as in \cite[\S 7.6]{CG}. When $\bq=q$ is the size of the residue field, it is well-known that $H|_{\bq=q}=C_c(I\rquotient G^\vee(F)/I, \Z)$ is the algebra of $I$-biinvariant compactly supported functions on $G^\vee(F)$.

\subsection{Action of field automorphisms}
\label{subsection action of field automorphisms}
Now we fix once and for all a field automorphism $\gamma\in\Aut(\C)$. 
Following \cite[\S 2]{KSV}, \cite[\S5]{SolleveldStandard}, let $\C_\gamma$ denote the $\C\text{-}\C$-bimodule $\C$ with actions 
\[
z_1\cdot z\cdot z_2=z_1z\gamma(z_2).
\]
Let $\pi$ be an admissible $G^\vee(F)$-representation, and let $\gamma\cdot\pi=\C_\gamma\otimes_\C\pi$ be the twist of $\pi$ by $\gamma$. It has scalar multiplication given by 
\[
z\cdot(1\otimes v)=z\otimes v=1\otimes\gamma(z)^{-1}v
\]
and $G^\vee(F)$-action given by 
\[
(\gamma\cdot\pi) g\cdot(1\otimes v)=1\otimes \pi(g)v. 
\]
The matrix coefficients of $\gamma\cdot\pi$ are obtained from the matrix coefficients of $\pi$ by applying $\gamma$. Indeed, clearly we have $\widetilde{\gamma\cdot\pi}=\gamma\cdot\widetilde{\pi}$, with pairing
\[
\pair{z_1\otimes\tilde{v}}{z\otimes v}_{\gamma\cdot\pi}=z_1z\gamma(\pair{\tilde{v}}{v}_{\pi}).
\]
This means that for $1\otimes v=1\otimes v_{\alpha_0}\in\gamma\cdot\pi$ for $\sett{v_\alpha}_\alpha$ some basis of $\pi$, and $1\otimes\tilde{v}\in\gamma\cdot\tilde{\pi}$, such that $\pair{\tilde{v}}{v}_\pi=1$, then writing 
\[
(\gamma\cdot\pi)(g)(1\otimes v)=1\otimes \pi(g)v=1\otimes\left(\sum_{\alpha}c_\alpha v_\alpha\right)=\sum_{\alpha}\gamma(c_\alpha)\otimes v_\alpha,
\]
we have
\[
\pair{1\otimes\tilde{v}}{(\gamma\cdot\pi)(g)(1\otimes v)}_{\gamma\cdot\pi}=
\pair{1\otimes\tilde{v}}{1\otimes\pi(g)v}_{\gamma\cdot\pi}=\gamma(c_\alpha)=\gamma\left(\pair{\tilde{v}}{\pi(g)v}_{\pi}\right).
\]

Given $f\in \mathcal{H}=C_c^\infty(G^\vee(F))$, the matrix coefficients of the operator $(\gamma\cdot\pi)(f)$ are then by definition the numbers
\begin{equation}
\label{HeckeMatCoeff}
\pair{1\otimes \tilde{v}}{(\gamma\cdot\pi)(f)(1\otimes v)}_{\gamma\cdot\pi}=
\dInt{f(g)\pair{1\otimes\tilde{v}}{(\gamma\cdot\pi)(g)(1\otimes v)}_{\gamma\cdot\pi}}{g}=
\dInt{f(g)\gamma\left(\pair{\tilde{v}}{\pi(g)v}\right)}{g}.
\end{equation}
As we normalize the Haar measure such that open compact subgroups of $G$ have rational volumes, if $f\in C_c^\infty(G,\Q)$ is rational-valued, then \eqref{HeckeMatCoeff} is equal to 
\[
\gamma\left(\pair{\tilde{v}}{\pi(f)v}_\pi\right),~~~f\in C_c^\infty(G^\vee(F),\Q).
\]
%

%Now let $(K,\rho)$ be a type such that $H(G,K,\rho)$ is isomorphic via $\Upsilon$ to a subalgebra of $\mathcal{H}$, which is isomorphic to an affine Hecke algebra $H$ with equal parameters. 
%%
%\begin{lem}
%Suppose that $\rho$ is realized over $\Q$, \textit{i.e.} $\rho$ factors via the inclusion $\GL_n(\Q)\into \GL_n(\C)$ for $n=\dim\rho$. Then $H(G,K,\rho)$ is defined over $\Q$.
%\end{lem}
%%
%\begin{proof}
%Realisability over $\Q$ of $\rho$ implies that it is well-defined to set
%%
%\[
%H(G,K,\rho)_\Q=\sets{f\colon G\to\End(\rho)(\Q)}{f\in H(G,K,\rho)}.
%\]
%%
%Clearly $H(G,K,\rho)=H(G,K,\rho)\otimes_\Q\C$.
%\end{proof}
%%
%Rescaling priemage in $H(G,K,\rho)$ of the basis $\sett{T_w}_{w\in\tilde{W}}$ of $H$ under the assumed isomorphism yields an isomorphism
%%
%\[
%H(G,K,\rho)_\Q\overset{\sim}{\to}H_\Q
%\]
%%
%where $H_\Q=\spn_\Q{\{T_w\}}_{w\in\tilde{W}}$ for $\tilde{W}$ the affine Weyl group corresponding to $H$. Let $G_H$ be the corresponding dual group over $\C$ for $\tilde{W}$, \textit{i.e.} the group appearing in the parametrization of simple (and standard) $H$-modules. 

\subsection{Iwahori-spherical standard modules as vector spaces}
Recall that the standard modules are indexed by triples $(s,N,\rho)$ up to $G(\C)$-conjugacy, where $s\in G(\C)$ is semisimple, $N\in\g$ is nilpotent, $\Ad(s)N=qN$, and $\rho$ is an irreducible representation of the component group $\pi_0(Z_{G}(N,s))$ appearing in $H_*(\Bb_N^{(s,q)})$, where we view $\Bb_N$ as a $G\times\Gm$-variety as above. Then the corresponding standard module is by definition, $K(s,N,\rho)=H_*(\Bb_N^s)_\rho$, the $\rho$-isotypic part of $H_*(\Bb_N^s)$. This is shown in \cite{KLDeligneLanglands} \cite{CG} under the assumption that the derived subgroup of $G$ is simply-connected, and extended to the general case in \cite[Thm. 3.5.4]{ReederIsogeny}.

The adjoint representation depends by definition only on the adjoint form $G_{\mathrm{ad}}$ of $G$, which, by Chevalley's theorem, is certainly defined over $\Q$, as is $\g$ (see, \textit{e.g.}, \cite{LusJAMS2009}). Therefore the adjoint representation of $G_{\mathrm{ad}}$ on $\g$ is defined over $\Q$ \cite[\S 34.2 (9.1), (10.3)]{HumphreysLAG}, and hence each nilpotent orbit of $G$ is defined over $\Q$ \cite[\S 34.2 (8.3)]{HumphreysLAG}. By the proof of Theorem 4.2 of \cite{Kottwitz}, each nilpotent orbit in $\g$ therefore has a 
$\Q$-point, so we may assume that $N$ is $\gamma$-fixed. Therefore we have
\begin{equation}
\label{diagram twist definition}
\begin{tikzcd}
(\Bb_N^s)_\gamma=\Bb_N^s\times_{\Spec\C}\Spec\C=\Bb_N^{\gamma(s)}
\arrow[d, "\Spec(\gamma)"]\arrow[r]&\Spec\C\arrow[d, "\Spec(\gamma)"]&\C\\
\Bb_N^s\arrow[r]&\Spec\C&\C\arrow[u, "\gamma"]
\end{tikzcd}
\end{equation}
where we abusively write $\Spec(\gamma)$ for both the original morphism and its base-change, a convention we will use throughout when there is no danger of confusion.

Diagram \eqref{diagram twist definition} exhibits $\Bb_N^s$ and $\Bb_N^{\gamma(s)}$ as isomorphic $\Q$-varieties,  but the two are not isomorphic $\C$-varieties unless $\gamma=\id_\C$, and $\gamma$ does not even induce a continuous map of analytifications unless it is complex conjugation or the identity. None the less, there is still a natural map on homology, thanks to
\begin{theorem}[\cite{DLP}]
\label{thm DLP}
For any $(s,N)$ such that $\Ad(s)N=qN$, we have an isomorphism of Chow groups and homology
\[
A_*(\Bb_N^s)\overset{\sim}{\to} H_*(\Bb_N^s;\Z)
\]
via the cycle class map.
\end{theorem}
Writing $A_*(X)_\Q=A_*(X)\otimes_\Z\Q$ for any scheme $X$, we therefore consider the standard modules as realized on $\C\otimes_\Q A_*(\Bb_N^s)_\Q$, and and define a $\Q$-linear isomorphism
\begin{equation}
\label{Spec gamma flat pullback eqn defn}
\Spec(\gamma)^*\colon A_*(\Bb_N^s)_\Q\overset{\sim}{\to} A_*(\Bb_N^{\gamma(s)})_\Q
\end{equation}
by flat pullback of algebraic cycles. We will occasionally refer to \cite{CG} for results written in the language of Borel-Moore homology, this is harmless by Theorem \ref{thm DLP}. Note that by \cite[Cor. 18.3.2]{Fulton}, the Chern character 
\[
\Ch_*\colon K(\Bb_N^s)\otimes_\Z\Q\overset{\sim}{\to} A_*(\Bb_N^s)_\Q
\]
is an isomorphism.
\subsection{Standard modules as modules over the Bernstein subalgebra}
\label{subsection standard as module over Bernstein}
We define an action of $G$-equivariant line bundles on the diagonal 
$\Delta\Nn\subset\St$ on $A_*(\Bb_N^s)_\Q$ based on the digression in \cite[\S 8.2]{CG}: Let $\Ll$ be such a line bundle, and consider its class as belonging to $K(\Bb)$ via the Thom isomorphism. Consider the natural embeddings
\[
\begin{tikzcd}
\Bb^s\arrow[r, "\iota", hook]&\Bb\\
\Bb^{\gamma(s)}=(\Bb^s)_\gamma\arrow[u,"\Spec(\gamma)"]\arrow[r, "\iota_\gamma", hook]& \Bb\arrow[u, "\Spec(\gamma)"],
\end{tikzcd}
\]
and let $v\in A_*(\Bb_N^s)_\Q$. Then 
\[
\Ll\cdot v:=\Ch^*(\iota^*\Ff)\frown v
\]
is the action of cohomology on Borel-Moore homology \cite[\S 2.6.40]{CG} and $\Ch^*$ is the cohomological Chern character \cite{BFMIHES} \cite[\S 5.8]{CG}.
\begin{lem}
The $\Q$-linear isomorphism $\Spec(\gamma)^*$ in \eqref{Spec gamma flat pullback eqn defn} is linear for the $\cdot$-action of line bundles.
\end{lem}
\begin{proof}
This is entirely formal but notationally involved. Write $\Spec(\gamma)$ for all the vertical maps in the pair of pullback squares
\[
\begin{tikzcd}
(\Bb_N^s)_\gamma\arrow[r, hook]\arrow[d]&(\Bb^s)_\gamma\arrow[d]\arrow[r]&\Spec\C\arrow[d]\\
\Bb_N^s\arrow[r, hook]&\Bb^s\arrow[r]&\Spec\C.
\end{tikzcd}
\]
Let $\Ee\in K(\Bb_N^s)$ be such that $\Ch_*(\Ee)=v$. 

We have, for any line bundle $\Ll$ on $\Delta\Nn$,
\begin{align}
\Spec(\gamma)^*(\Ll\cdot v)&=\Spec(\gamma)^*(\Ch^*(\iota^*\Ll)\frown\Ch_*(\Ee))
\label{eqn module1 pre}
\\
&=
\Spec(\gamma)^*(\Ch_*(\iota^*\Ll\otimes\Ee))
\label{eqn module1 post}
\\
&=
\Ch_*(\Spec(\gamma)^*\iota^*\Ll\otimes\Spec(\gamma)^*\Ee)
\label{eqn PB1}
\\
&=
\Ch_*(\iota_\gamma^*\Spec(\gamma)^*\Ll\otimes\Spec(\gamma)^*\Ee)
\nonumber
\\
&=
\Ch^*(\iota_\gamma^*\Spec(\gamma)^*\Ll)\frown\Ch_*(\Spec(\gamma)^*\Ee)
\label{eqn module2 pre}
\\
&=
\Ch^*(\iota_\gamma^*\Ll)\frown\Spec(\gamma)^*v,
\label{eqn module2 post}
\end{align}
where between \eqref{eqn module1 pre} and \eqref{eqn module1 post} we used smoothness of $\Bb^s$ and \cite[Prop. 18.1 (c)]{Fulton}, and likewise between \eqref{eqn module2 pre} and \eqref{eqn module2 post}; between \eqref{eqn module1 post} and \eqref{eqn PB1} and \eqref{eqn module2 pre} and \eqref{eqn module2 post}, we also used that the homological Chern character commutes with pullback by smooth proper morphisms \cite[(2.6)]{BFMIHES}, 
\cite[Thm. 18.1]{Fulton}. Between \eqref{eqn module2 pre} and \eqref{eqn module2 post} we used that $\Spec(\gamma)^*[\Ll]=[\Ll]$ in the Grothendieck group, as the associated bundle construction defining each generator $\Oo_{\lambda}$ of $K(\Nn/G\times\Gm)$, $\lambda\in X^*(T)$, is defined over $\Q$ (in particular this implies the equality in equivariant $K$-theory, see also Step 1 of the proof of Proposition \ref{prop big diagram}). Between 
\eqref{eqn module1 post} and \eqref{eqn PB1} we used that $\Spec(\gamma)$ is flat and $\iota^*\Ll$ is a line bundle.
\end{proof}
%\todo[inline]{for pullback and tensor product commuting, note that $\iota^*\Ll$ is a vector bundle, and $\gamma^*\iota^*$ is a vector bundle, so the tensor product is the underived tensor product for which this formula is true.}

Now we recall the definition of the $H$-module structure on $A_*(\Bb_N^s)_\C$ from \cite[\S 8.2]{CG}. 
Chose a reference Borel subgroup $B_0$ of $G_{\mathrm{ad}}$ such that $\Bb=G_{\mathrm{ad}}/B_0$. We may chose $B_0$ to be defined over $\Q$, so that it is $\gamma$-stable. Let $B$ be another Borel subgroup of $G_{\mathrm{ad}}$, so that there is $g\in G(\C)$ such that $B=gB_0g^{-1}$. By \cite[6.1.1]{CG}  the induced isomorphism $\varphi_B$ 
\begin{equation}
\label{diag lambda transport}
\begin{tikzcd}
B\arrow[r, "g^{-1}(-)g"]\arrow[d, two heads]&B_0\arrow[d, two heads]\\
B/[B,B]\arrow[r, "\varphi_B"]& B_0/[B_0,B_0]=:\mathbb{T}
\end{tikzcd}
\end{equation}
of tori, depends only on $B$, not on $g$. Now for $\Oo_\lambda$ on $\Delta\Nn\subset\St$, the action of $\Oo(\lambda)$ on $\C\otimes_\Q A_*(\Bb_N^s)_{\Q}$ is defined in \cite[\S 8.2]{CG} as follows:
Let $v$ be a linear combination of cycles contained in a single connected component of $\Bb_N^s$. This component is contained in a unique connected component $\Bb_j$ of $\Bb^s$, and $\Bb_j=Z_{G}(s)\cdot B_j$ for some Borel subgroup $B_j$ \cite[8.8.7]{CG}. As $s\in B_j$ and other Borel subalgebras in $\Bb_j$ are $Z_{G}(s)$-conjugate, it is well-defined to put
\[
\Oo(\lambda)\star (1\otimes v):=\lambda(\varphi_{B_j}(s))\otimes(\Oo(\lambda)\cdot v).
\]
The action is then extended to the entire Chow group by linearity.
\begin{dfn}
\label{def Phi(gamma)}
Let $\Phi(\gamma)$ be the isomorphism of left $\C$-modules
\[
\Phi(\gamma)\colon \C_\gamma\otimes_\C\left(\C\otimes_\Q A_*(\Bb_N^s)_{\Q}\right)\to \C\otimes_\Q A_*(\Bb_N^{\gamma(s)})_\Q
\]
given by 
\[
\left(z_1\otimes\left(z_2\otimes v\right)\right)\mapsto z_1\gamma(z_2)\otimes\Spec(\gamma)^*v.
\]
\end{dfn}
The automorphism $\gamma$ induces an isomorphism 
\begin{equation}
\label{component group isomorphism}
\gamma\colon Z_{G}(s,N)\to Z_{G}(\gamma(s),N),
\end{equation}
and hence an isomorphism of component groups $\pi_0(Z(s,N))\simeq \pi_0(Z(\gamma(s),N))$, such that the diagram
\begin{equation}
\label{diagram component}
\begin{tikzcd}
\Bb_N^{\gamma(s)}\arrow[r, "\gamma(g)"]\arrow[d, "\Spec(\gamma)"]
&
\Bb_N^{\gamma(s)}\arrow[r]\arrow[d, "\Spec(\gamma)"]
&\Spec(\C)\arrow[d, "\Spec(\gamma)"]
\\
\Bb_N^s\arrow[r, "g"]&\Bb_N^s\arrow[r]&\Spec\C
\end{tikzcd}
\end{equation}
commutes.
\begin{lem}
\label{lem linearity component group}
The map $\Phi(\gamma)$ is $\pi_0(Z_G(N,s))$-linear, upon identifying component groups via \eqref{component group isomorphism}.
\end{lem}
\begin{proof}
This is exactly commutativity of \eqref{diagram component}.
\end{proof}
\begin{lem}
\label{BernsteinAction}
The map $\Phi(\gamma)$ intertwines, for each $\Oo_\lambda$ the actions 
\[
\Oo_\lambda(z_1\otimes(z_2\otimes v))=z_1\otimes(\Oo_\lambda\star(z_2\otimes v))
\]
and the $\star$-action defined above on $A_*(\Bb_N^s)_\C$.
\end{lem}
\begin{proof}
If $B_j=gB_0g^{-1}$ is in the connected component of $\Bb^s$ containing the component of $\Bb_N^s$ in which a cycle $v$ is contained, then $B_{\gamma(j)}:=\gamma(g)\gamma(B_0)\gamma(g)^{-1}=\gamma(g)B_0\gamma(g)^{-1}$ is in the connected component of $\Bb^{\gamma(s)}$ in which $\gamma^*v$ is contained. Hence we have
\[
\lambda(\varphi_{B_{\gamma(j)}}(\gamma(s)))=\lambda(\gamma(g)\gamma(s)\gamma(g)^{-1})=\gamma(\lambda(gsg^{-1}))=\gamma(\lambda(\varphi_B(s))
\]
as required. The Lemma is true for general cycle classes by linearity.
\end{proof}
\subsection{Standard and essentially square-integrable representations}
Our first main result is the following theorem and its corollary.
\begin{theorem}
\label{thm Hecke twist}
The map $\Phi(\gamma)$ is $H$-linear, and commutes with the action of the component group 
of $Z_G(s,N)\simeq Z_G(\gamma(s),N)$. Hence we obtain an isomorphism
\[
\Phi(\gamma)\colon\C_\gamma\otimes_\C K(s,N,\rho)\overset{\sim}{\to}K(\gamma(s),N,\rho),
\]
$H$-modules, where the action on the left hand side is via 
\[
h\cdot(z_1\otimes v)=z_1\otimes h\cdot v.
\]
\end{theorem}
\begin{proof}
Combine Proposition \ref{prop big diagram} below and Lemma \ref{lem linearity component group}.
\end{proof}
\begin{rem}
\label{rem EM}
For generic $s$ and $\gamma(s)$, the standard modules $K(s,N,\rho)$ and $K(\gamma(s),N,\rho)$ are simple. For $G$ simply connected, by \cite[Thm. 5.5]{EvensMirkovic}, Lemma \ref{BernsteinAction} already implies that $K(s,N,\rho)\simeq K(\gamma(s),N,\rho)$ in this case.
\end{rem}
\begin{dfn}
A smooth admissible representation of $\omega$ of $G^\vee(F)$ is \emph{square-integrable modulo the centre} $Z(G^\vee(F))$ if $\omega$ admits a unitary central character and all matrix coefficients of $\omega$ are square-integrable functions on $G^\vee(F)/Z(G(F))$.

A smooth admissible representation $G^\vee(F)$ is 
\emph{essentially square integrable} if it is of the form $\omega\otimes\nu$ for an unramified character $\nu$ and square-integrable modulo centre representation $\omega$ of $G^\vee(F)$.
\end{dfn}
Recall that a standard representation $K(s,N,\rho)$ is an irreducible essentially square-integrable representation, if and only if, there is no proper parabolic subgroup of $G$ with Levi subgroup containing $(s,u)$, where $u=\exp(N)$ \cite[Thm. 8.3 (a)]{KLDeligneLanglands}. 
As recalled in \cite [\S 2.1]{Reeder}, this is equivalent to $Z_G(s)$ being semisimple with $N$ a distinguished nilpotent in $\Lie(Z_G(s))$. Hence the eigenvalues of $s$ are roots of unity by Kac's classification.
\begin{cor}
\label{Cor Hecke twist standard and formal degrees}
Let $\pi$ be an essentially square-integrable representation of $H$. Then $\gamma\cdot\pi$ is essentially square-integrable, and we have an equality of formal degrees
\[
d(K(s,N,\rho)=d(K(\gamma(s),N,\rho).
\]
More generally, if $\pi$ is a standard $H$-module, so is $\gamma\cdot\pi$, and twist by $\gamma$ also preserves genericity of the Langlands quotients (but with respect to the Whittaker datum for the $\gamma$-twisted generic character.
\end{cor}
\begin{proof}
The last statement is a restatement of Theorem \ref{thm Hecke twist}.

If $\pi$ is square-integrable, then $\pi=K(s,N,\rho)$ for some $\rho$. Clearly, for $u$ as above, $(s,u)$ is not contained in any proper Levi subgroup of $G$ if and only if same is true for $(\gamma(s),u)$. Therefore the last statement follows from Theorem \ref{thm Hecke twist}.

Now we argue for the equality of formal degrees. Recall from \cite[Formula A]{Reeder}, that we have 
\[
d(K(\gamma(s),N,\triv))^{-1}=q^{\dim\Bb}\sum_{J\subset\Sigma}\frac{1}{P_J(q)}\sum_{\lambda\in\Lambda_J}|M(\lambda, \gamma(s))|^2,
\]
where $W_J$ is the parabolic subgroup of $W$ generated by $J\subset\Sigma$, $P_J(q)$ is its Poincar\'{e} polynomial, 
\[
\Lambda_J=\sets{\lambda\in X^{*,+}}{\stab{W}{\lambda}=W_J},
\]
where $X^{*,+}$ is the set of dominant weights, and $M(\lambda,s)\colon\Lambda_J\to\C$ is the function
\[
M(\lambda,s)=\sum_{w\in W}w\left(\lambda\prod_{\alpha>0}\frac{1-q^{-1}\alpha}{1-\alpha}\right)(s).
\]
The product is taken over all positive roots, and the action of the Weyl group is the natural one on rational functions on $T$. 

Recall that from $(s,N)$ $q$-commuting we can obtain commuting $(s_1,N)$: Let  $\phi$ be a homomorphism 
\[
\phi\colon\SL_2\to G
\]
obtained from $N$ via the Jacobson-Morozov theorem. Then $s=s_1\phi(\diag(q^{-1/2},q^{1/2}))$ and $\Ad(s_1)N=N$.

Now, in general, $\gamma$ of course does not commute with complex conjugation. In the case at hand, though, square-integrability of $K(s,N,\rho)$ and $K(\gamma(s),N,\rho)$ implies that $s_1$ and $\gamma(s_1)$ belong to the maximal compact subgroup of $T^\vee$. Therefore 
\[
\lambda(\gamma(s))=\lambda(\gamma(s_1))\lambda(\phi(\pm\diag(q^{-1/2},q^{1/2}))
\]
with $\lambda(\gamma(s_1))\in U(1)\subset\C^\times$ for any weight $\lambda$, and, by hypothesis, likewise for $s_1$, noting that
\[
\gamma(s_1)=\gamma(s)\phi(\gamma(q^{-1/2},q^{1/2})=\gamma(s)\phi(\pm\diag(q^{-1/2},q^{1/2}).
\]
Therefore
\begin{multline*}
\overline{\lambda(w\gamma(s))}=
\lambda(w\gamma(s_1))^{-1}
\lambda(\phi(\pm\diag(q^{-1/2},q^{1/2}))
\\
=
\lambda(w\gamma(s_1^{-1}))\lambda(\phi(\pm\diag(q^{-1/2},q^{1/2}))
=\gamma\left(\lambda(ws_1^{-1})\right)\lambda(\phi(\pm\diag(q^{-1/2},q^{1/2}))
=\gamma\left(\overline{\lambda(ws)}\right).
\end{multline*}
It follows that
\[
|M(\lambda,\gamma(s))|^2=
M(\lambda, \gamma(s))\overline{M(\lambda,\gamma(s))}=\gamma\left(|M(\lambda, s)|^2\right),
\]
and so $d(K(\gamma(s),N,\triv))=\gamma(d(K(s,N,\triv))$. On on the other hand, 
by \cite{FOS}, $d(K(s,N,\rho))\in\Q$. For general $\rho\neq \triv$, we have 
\[
d(K(s,N,\rho))=\dim(\rho)d(K(s,N,\triv))
\]
by \textit{op. cit}.

%we obtain 
%%
%\[
%|M(\lambda,\gamma(s))|^2=M(\lambda,\gamma(s))\overline{M(\lambda,\gamma(s))}=
%\gamma
%\]

\end{proof}
\begin{cor}
\label{Cor DL}
If $L_{s,N,\rho}$ is the simple $H_I$-module associated to $(s,N,\rho)$ by \cite{KLDeligneLanglands} or \cite{CG}, then $\gamma\cdot L_{s,N,\rho}=L_{\gamma(s),N,\rho}$.
\end{cor}
\begin{proof}
By definition, $L_{s,N,\rho}$ is the unique simple quotient $H$-module of $K(s,N,\rho)$. By exactness of $\gamma\cdot(-)$, $\gamma\cdot K(s,N,\rho)$ has unique simple quotient $\gamma\cdot L(s,N,\rho)$. By Corollary \ref{Cor Hecke twist standard and formal degrees}, $\gamma\cdot K(s,N,\rho)=K(\gamma(s),N,\rho)$.
\end{proof}
%
%\begin{rem}
%\label{rem genericity}
%Note that $\gamma\cdot L_{s,N,\triv}=L_{\gamma(s),N,\triv}$ is again generic, but now with respect to the Whittaker datum obtained from the $\gamma$-twist of the original generic character.
%\end{rem}
%
%%
%
%
%The arguments of the preceeding sections have proven
%%
%
%%
%\begin{cor}
%an $H$-module $\pi$ is essentially discrete series if and only if $\gamma\cdot\pi$ is.
%\end{cor}
%%
%\begin{proof}
%Every essentially discrete series module is of the form $K(s,N,\rho)$, where $(s,N)$ are of the following form: There is $s_1$ a semisimple element of $G_H$ with semisimple centralizer $Z_s$
%\todo[inline]{for essentially tempered, should be semisimple modulo centre of $G$ or something}
%, such that $N$ is a distinguished nilpotent in $\Lie(Z_s)$. If $\phi\colon\SL_2\to G_H$ is obtained from $N$ by the Jacobson-Morozov theorem, then we require $s=s_1\phi(\diag(q^{-\frac{1}{2}}, q^{\frac{1}{2}}))$. 
%\end{proof}
%
%
%\begin{rem}
%Geometric constructions of standard modules of affine Hecke algebras are given in entirely $K$-theoretic terms in \cite{KLDeligneLanglands} and \cite{CellsIV}. However, treatment via Borel-Moore homology in \cite{CG} seems more adapted to the present problem; fixing the central character immediately, usually a disadvantage of constructible techniques, saves one from having to pull back sheaves between equivariant $K$-theory for different groups.
%\end{rem}
\section{Essentially discrete series representations of inner forms of $\GL_n$}
\subsection{Types and Hecke algebras}
\label{Types and hecke algebras}
In Section \ref{Types and hecke algebras} only, we reverse our conventions on $G$ vs $G^\vee$.

In this section, we give our proof of Theorem \ref{thm Clozel} for $G$ an inner form of $\GL_n$. We first recall
\begin{theorem}[\cite{BushKutzBook}, \cite{BushKutzComp}]
\label{thm BushnellKutzko}
Let $F$ be a non-archimedean local field and let $G=\GL_n(F)$; let $q$ be the cardinality of the residue field of $F$. Then
\begin{enumerate}
\item 
For any Bernstein block $\Rep_{\sS}(G)$ of the category of smooth representations $\Rep(G)$ of $G$, there are integers $\sett{m_i}_i$ and $\sett{f_i}_i$ and an equivalence of categories
\begin{equation}
\label{eqn BK type equivalence}
\Rep_{\sS}(G)\simeq\bigotimes_{i} H_{m_i}(q^{f_i})-\Mod,
\end{equation}
for integers $m_i$ and $f_i$, where $H_{m_i}(q^{f_i})$ is an affine Hecke algebra of type 
$\tilde{A}_{m_i}$ with parameter $q^{f_i}$, and the tensor product is taken over $\C$. 
\item
If $\Rep_{\sS}(G)$ contains an essentially discrete series representation, then on the right hand side of \eqref{eqn BK type equivalence} appears a single Hecke algebra $H_{m}(q^f)$, and the equivalence \eqref{eqn BK type equivalence} induces a bijection of essentially discrete series representations preserving formal degrees up to explicit scalar factors.
\end{enumerate}
\end{theorem}
The statements of the second part are Corollary 8.5.11, Theorem 7.7.1, and Corollary 7.7.11 of \cite{BushKutzBook}, respectively.

Now let  $G$ be an inner form of $\GL_n$, so that $G(F)=\GL_n(D)$ for a division algebra $D$ with centre $F$.

%\todo[inline]{when block multiplicities are the same, get Morita equivalence of J-rings, whereas}
%
\begin{theorem}[\cite{Secherre}, \cite{SecherreStevens}]
\label{thm SecherreStevens}
Let $G$ be an inner form of $\GL_n$ and $L$ be a Levi subgroup, so that $L(F)=\prod_{i}\GL_{n_i}(D)^{e_i}$ with $\sum_{i}n_ie_i=n$ and let $\omega=\bigotimes_i\omega_i^{\otimes e_i}$ be a supercuspidal representation of $L(F)$ for $\omega_i$ a supercuspidal representation of $\GL_{n_i}(D)$, grouped such that $\omega_i$ is not inertially equivalent with $\omega_j$ if $i\neq j$. Let $\sS=[L,\omega]$ be the corresponding inertial equivalence class and $\Rep(G)_{\sS}$ the resulting Bernstein block. Then 
\begin{enumerate}
\item 
There is an equivalence of categories
\begin{equation}
\label{eqn SeSt equiv}
\Rep(G)_{\sS}\simeq\bigotimes_i H_{e_i}(q^{f_i})-\Mod
\end{equation}
for integers $f_i$, where $H_{e_i}(q^{f_i})$ is an affine Hecke algebra of type $\GL_{e_i}$ with parameter $q^{f_i}$, and the tensor product is taken over $\C$. 
\item
If $\Rep_{\sS}(G)$ contains an essentially discrete series representation, then $\sS$, then on the right hand side of \eqref{eqn SeSt equiv} appears a single Hecke algebra $H_{e}(q^f)$, and the equivalence \eqref{eqn BK type equivalence} induces a bijection of essentially discrete series representations preserving formal degrees up to explicit scalar factors.
\end{enumerate}
\end{theorem}
The fact that essentially square-integrable representations can only contain simple types follows from \cite{Secherre}, \cite{SecherreStevensIV} (c.f. \cite[p.3]{SecherreV}). As noted by Solleveld \cite{SolleveldCompletion}, the remaining statements follow as in \cite[\S 7]{BushKutzBook} (c.f. the discussion at the start of Section 7.7 of \textit{loc. cit.}), once one knows that the isomorphism of affine and typical Hecke algebras behind \eqref{eqn SeSt equiv} preserves the $*$-involution and the trace, as with these two ingredients one builds the pairing after (7.7.3) of \textit{loc. cit.}. This preservation is checked by Solleveld \cite[Thm. 4.3]{SolleveldCompletion}.

%
%\begin{theorem}[\cite{BHK}, Thm. 4.3]
%\label{thm BHK}
%The equivalences of Theorems \ref{thm BushnellKutzko}, \ref{thm SecherreStevens} preserve Plancherel measures up to a constant factor, and hence preserve discrete series and tempered representations.
%\end{theorem}

%\todo[inline]{For simple types, formal degrees and preservation of discrete series is Bushnell-Kutzko \S 7. Try to find in Secherre-Stevens.}
%
%\todo[inline]{possibly remove therefore BHK theorem, and just say that holds for general types, because we don't need it.}

As remarked in the introduction, if $\sigma$ is a supercuspidal representation of $G(F)$, then clearly $\gamma\cdot\sigma$ is again supercuspidal. The compatibility with normalized parabolic induction, up to unramified twist of the induced representation, of $\gamma\cdot(-)$ \cite[(5.12)]{KSV} then implies that the exact auto-equivalence $\pi\mapsto\gamma\cdot\pi$ of $\Rep(G)$ sends the Bernstein block for an inertial class $\sS=[L,\sigma]_G$ to the Bernstein block for 
$\sS_\gamma=[L,\gamma\cdot\sigma]_G$. As remarked above, $[T,\triv]_G=[T,\gamma\cdot\triv]_G$, but in general $\sS_\gamma$ labels a new block. 
Note that $\sS$ is homogeneous if and only if $\gamma\cdot\sS$ is.

If $(K,\rho)$ is a type for $\sS$, then $(K,\gamma\cdot\rho)$ is a type for $\sS_\gamma$, and we have an isomorphism of rings
\[
H(G,K,\rho)\to H(G,K,\gamma\cdot\rho)
\]
given by 
\[
f\mapsto (g\mapsto \gamma(f(g)),
\]
which is $\gamma$-semilinear, where $\gamma(f(g))$ means we apply $\gamma$ to the matrix of coefficients of $f(g)\in\End_\C(\rho)$. Equivalently, we have a $\C$-linear isomorphism
\[
\C_\gamma\otimes_\C H(G,K,\rho)\to H(G,K,\gamma\cdot\rho)
\]
given by 
\[
z\otimes f\mapsto z\gamma(f).
\]

Recall that the construction of types in \cite{BushKutzBook}, \cite{BushKutzComp}, \cite{Secherre} \cite{SecherreStevens} is done with respect to a chosen additive character $\psi_F\colon F\to\C^\times$, but that neither the notion of essentially square-integrable $G(F)$-representation, essentially-square-integrable module over an affine Hecke algebra, refer to this choice, and likewise with standard representations and standard modules.
\begin{lem}
\label{lem types}
The type $(K, \gamma\cdot\rho)$ is again a type obtained by the constructions of Theorem \ref{thm BushnellKutzko} and \ref{thm SecherreStevens}, but with respect to the additive character $\gamma\circ\psi$. (c.f. Corollary \ref{Cor Hecke twist standard and formal degrees}.)
\end{lem}
\begin{proof}
We adopt the terminology and notation of \cite{BushKutzBook}, \cite{Secherre}.

As in the proof of Theorem \ref{thm SolleveldCompletion}, it is enough to check this for simple types, corresponding to a single tensor factor in either of Theorems \ref{thm BushnellKutzko} or \ref{thm SecherreStevens}.
Hence we consider simple strata $[\mathfrak{A}, n,0,\beta]$. The definition of simplicity refers to $\psi$, but the sets of simple strata are the same for $\psi$ and $\psi\circ\gamma$. Note also that the $*$- and $\dagger$-duals of Sections 1.1.4 and 1.3.6 of \cite{BushKutzBook}, respectively, are the same for $\psi$ and $\gamma\circ\psi$, and that Section 1.3 of \textit{op. cit.} depends only on the (unchanged by $\gamma)$ conductor of $\psi$, and the coherences provided by Theorem 3.6.1 of \textit {op. cit.} are $\Aut(\C)$-equivariant.

%\todo[inline]{on one hand seems clear from BK that Coxeter presentation of affine Weyl group for simple type depends only on the stratum in an $\Aut(\C)$ invariant way. On the other hand one side of the Hecke algebra isomorphism has to have something to do with $\lambda$, which changes appropriately as checked below.}
%
%\todo[inline]{but point should just be to check that everything is $\Aut(\C)$-invariant and then appeal to Solleveld.}

%we need to check this for types for supercuspidal representations. That is, in the notation of 
%\cite{SecherreI}, \cite{BushKutzBook}, consider a supercuspidal representation $\sigma=\mathrm{c-Ind}_{\bar{J}}^G(\Lambda)$, where $\Lambda$ is an extension of a representation $\lambda$ of the open compact subgroup $J$, where 
Now, the representation $\rho$ is constructed as
$\rho=\kappa\otimes\sigma_0$ is the tensor product of a homogeneous cuspidal representation of the finite group of Lie type obtained a quotient of a certain open compact subgroup $J=J(\beta, \mathfrak{A})$, and a character $\kappa$ obtained as follows: First, $\kappa$ is an extension to $J$ of a representation $\eta(\theta)$, in turn an extension of a simple character (in the sense of \cite{BushKutzBook}) $\theta$ of the group $H^1(\beta,\mathfrak{A})$.

The characterization of $\kappa$ in terms of $\eta$ in terms of intertwining sets implies that $\gamma\cdot\kappa(\eta)=\kappa(\gamma\cdot\eta)$, and in turn the unicity of $\eta(\theta)$ implies 
$\gamma\cdot\eta(\theta)=\eta(\gamma\cdot\theta)$. Tracing the definition of special character back though \cite{SecherreI} to \cite[(3.2.1)]{BushKutzBook}, we see that if $\theta$ is a special character of $H^{m+1}(\beta)$ defined with respect to $\psi$, then $\gamma\cdot\theta$ is a simple character of the same subgroup for the same stratum, only with respect to $\gamma\circ\psi$.

Clearly cuspidality of $\sigma_0$ is preserved by $\gamma\cdot(-)$.
%\todo[inline]{check through}
%\todo[inline]{think: notion of positivity of unramified character should not depend on choice of additive character, depends on choice of Borel...}
\end{proof}

Therefore to both $H(G,K,\rho)$ and $H(G,K,\gamma\cdot\rho)$ we can apply
\begin{theorem}[\cite{SolleveldCompletion}, p. 43]
\label{thm SolleveldCompletion}
The equivalences of categories in Theorem \ref{thm BushnellKutzko} and \ref{thm SecherreStevens} preserve standard representations.
\end{theorem}
We emphasize that this is a non-trivial condition because in general the root datum $\mathcal{R}$ associated to some $H(G,K,\rho)$ is not just that of $G$.

Now suppose that $\pi$ is a standard representation in $\Rep(G)_{\sS}$. By Theorem \ref{thm SolleveldCompletion}, $\pi$ corresponds to a $H(G,K,\rho)$-module
\begin{equation}
\label{eqn pi H-module GLN}
\bigotimes_i K(s_i,N_i)
\end{equation}
and $\gamma\cdot\pi$ corresponds to the $H(G,K,\gamma\cdot\rho)$-module
\begin{equation}
\label{eqn gamma pi H-module GLN}
\C_\gamma\otimes_\C\bigotimes_i K(s_i,N_i),
\end{equation}
by the discussion in Section \ref{section standard modules}.
Now we give our proof of Theorem \ref{thm Clozel}.

\subsection{Proof of Theorems \ref{thm standard} and \ref{thm Clozel} for inner forms of $\GL_n$}
We return to our convention that $G^\vee$ is the $p$-adic group.

To show that $\gamma\cdot\pi$ is again standard, it suffices by Section \ref{Types and hecke algebras} to show that the $\C_\gamma\otimes_\C\bigotimes_i H_{m_i}(q^{f_i})$-module \eqref{eqn gamma pi H-module GLN} is isomorphic via $\bigotimes_i\Phi(\gamma)$ to $\bigotimes_i K(\gamma(s_i),N_i)$ under the second isomorphism above to $\bigotimes_i H_{m_i}(q^{f_i})$. Therefore, it suffices to prove
%
%\todo[inline]{top pullback $\Spec(\gamma)^*$ seems not $G$-equivariant? because also act on $g$. However, on every generator of the Hecke algebra we have $\Spec(\gamma)^*[\Ff]=[\Ff]$ as sheaves, surely. We want to somehow not change the $G$-equivariant structure. Remember this map is supposed to be $z\mapsto z\gamma(f)$.}
%
%\todo[inline]{WANT TO SAY THAT ``EQUIVARIANT STRUCTURE IS DEFINED OVER $\Q$"}
\begin{prop} 
\label{prop big diagram}
Let $H$ be an affine Hecke algebra. Then the diagram 
\begin{center}
\begin{tikzcd}
\C_\gamma\otimes_\C\C\otimes_\Q K(\St/G\times\Gm)_\Q
\arrow[d, "i_s^*\circ\Forg_{\bangles{s}}"]
\arrow[rrr, "\id\cdot\gamma\otimes\Spec(\gamma)^*"]
&&&
\C\otimes_\Q K(\St/G\times\Gm)_\Q
\arrow[d, "i_{\gamma(s)}^*\circ\Forg_{\bangles{\gamma(s)}}"]
\\
\C_\gamma\otimes_\C\C\otimes_\Q K(\pt/\bangles{s})\otimes_\Q K(\St^{(s,q)})_\Q
\arrow[rrr, "\id\cdot\gamma\otimes\gamma\otimes\Spec(\gamma)^*"]
\arrow[d,"\mathrm{ev}"]
&&&
\C\otimes_\Q K(\pt/\bangles{\gamma(s)})\otimes_\Q K(\St^{\gamma(s),q)})
\arrow[d, "\mathrm{ev}"]
\\
\C_\gamma\otimes_\C\left(\C_a\otimes_\Q K(\St^{s,q})_\Q\right)
\arrow[d, "\id\otimes\left(-\otimes \left(1\boxtimes\lambda_s^{-1}\right)\right)"]
\arrow[rrr, "(\id\cdot\gamma)\otimes\Spec(\gamma)^*"]
&&&
\C_{\gamma(s)}\otimes_\Q K(\St^{\gamma(s),q})_\Q
\arrow[d, "\id\otimes\left(-\otimes \left(1\boxtimes\lambda_{\gamma(s)}^{-1}\right)\right)"]
\\
\C_\gamma\otimes_\C\left(\C_a\otimes_\Q K(\St^{s,q})_\Q\right)
\arrow[rrr, "(\id\cdot\gamma)\otimes\Spec(\gamma)^*"]
\arrow[d, "\id\otimes\id\otimes\mathrm{RR}"]
&&&
\C_{\gamma(s)}\otimes_\Q K(\St^{\gamma(s),q})_\Q
\arrow[d, "\id\otimes\id\otimes\mathrm{RR}"]
\\
\C_\gamma\otimes_\C\C_s\otimes_\Q A_*(\St^s)_\Q
\arrow[rrr, "\id\cdot\gamma\otimes\Spec(\gamma)^*"]
\arrow[d, "\id\otimes\id\otimes\mathrm{act}(v)"]
&&&
\C_{\gamma(s)}\otimes_\Q A_*(\St^{\gamma(s)})_\Q
\arrow[d, "\id\otimes\id\otimes\mathrm{act}\left(\Phi\left(\gamma\right)v\right)"]
\\
\C_\gamma\otimes_\C\C_s\otimes_\Q A_*(\Bb_N^s)_\Q
\arrow[rrr, "\Phi(\gamma)"]
&&&
\C_{\gamma(s)}\otimes_\Q A_*(\Bb_N^{\gamma(s)})_\Q,
\end{tikzcd}
\end{center}
where the columns encode the action of a single $H_{m_i}(q^{f_i})$ on a single $K(s_i,N_i)$ as per \cite[\S 8.2]{CG}, commutes. (Here, we have written $q=q^{f_i}$, $s=s_i$, etc. to unburden notation.)
\end{prop}
In the diagram, we write $\bangles{s}$ for the smallest closed subgroup of $G$ containing $s$.
Note also that that $H_*(\St^s;\Q)$ and $H_*(\St^{\gamma(s)};\Q)$ are isomorphic to 
$A_*(\St^s)_\Q$ and $A_*(\St^{\gamma(s)})_\Q$, respectively. This follows from \cite[Cor. 18.3.2] {Fulton} and \cite[Thm. 6.2.4]{CG} (the proof in \textit{loc. cit.} works for $\Q$-coefficients, by \cite[Rem. 8.1.10]{CG}). The K\"unneth formula for homology and \cite[(5.2.4)]{CG} reduce the general case to products of the case of the Proposition.

Of course, in the case of essentially square-integrable representations, only a single tensor factor appears, and Proposition \ref{prop big diagram} also implies the statement about formal degrees, as formal degrees of essentially discrete series representations of affine Hecke algebras of type $\GL_r$ depend only on $N$, and all of the scalar factors in \cite[Cor. 7.7.11]{BushKutzBook} are $\Aut(\C)$-fixed.

\begin{proof}
We explain the meaning and commutativity of each box in the diagram. Modulo notation, the proof is almost entirely formal: at least on the tensor factors involving $K$-theory or Chow groups, it essentially amounts to various constructions commuting with isomorphisms of schemes, \textit{i.e.}, being well-defined. 

%We write $\bangles{s}$ for the smallest closed subgroup of $G_H$ containing $s$.
%
%First, note that that $H_*(\St^s;\Q)$ and $H_*(\St^{\gamma(s)};\Q)$ are isomorphic to 
%$A_*(\St^s)_\Q$ and $A_*(\St^{\gamma(s)})_\Q$, respectively. This follows from \cite[Cor. 18.3.2] {Fulton} and \cite[Thm. 6.2.4]{CG} (the proof in \textit{loc. cit.} works for any field coefficients). 

It is enough to check commutativity of the square on elements of the form $1\otimes 1\otimes \Ff$ for generators $\Ff$ of $K(\St/G\times\Gm)$. The bundles $\Oo_\lambda$ on the diagonal component $\Delta\Nn\subset\St$, pulled back from the line bundles $\Oo(\lambda)$ on $\Bb$, generate the Bernstein subalgebra. We checked commutativity for the Bernstein subalgebra in Section \ref{subsection standard as module over Bernstein}.

We may take the following classes as generators of the finite Hecke algebra $H_W$, following \cite[(7.6.1)]{CG}: let $\bar{Y}_\sigma\subset\Bb\times\Bb$ be the Schubert variety indexed by the simple reflection $\sigma$ in $W$. The projection $\bar{Y}_\sigma\to\Bb$ to the first coordinate is smooth with $\Pp^1$ fibres, hence the relative cotangent bundle is a line bundle on $\bar{Y}_\sigma$, namely $\Omega^1_{\bar{Y}_\sigma/\Bb}=\Ind(P_\alpha/B\times^B\C_{-\alpha})$, where $\Ind\colon\Coh(P_\alpha/B)\to \Coh(\bar{Y}_\sigma/G)$ is the induction functor of \cite[\S 5.2.16]{CG} and $\alpha$ is the simple root corresponding to $\sigma$. Finally, our generator is $\Ll_\sigma:=\pi_\sigma^*\Omega^1_{\bar{Y}_\sigma/\Bb}$, where $\pi_\sigma\colon T_{\bar{Y}_\sigma}^*(\Bb\times\Bb)\to \bar{Y}_\sigma$ is the conormal bundle projection.

\begin{ex}
If $G=\SL_2$, then $\Ll_\sigma=\Oo\boxtimes\Oo(-2)$ on the zero section component $\Pp^1\times\Pp^1$ of $\St$.
\end{ex}

Now we check commutativity.
\begin{enumerate}
\item 
We check the topmost square, which subsumes the first three isomorphisms in \cite[(8.1.6)]{CG}. The morphisms $i_s$ and $i_{\gamma(s)}$ are inclusion of $s$-, respectively $\gamma(s)$-fixed points, and the forgetful maps are induced by the corresponding obvious functors. 

Taking a generator $\Ll_\sigma$, the character that appears in the $K(\pt/\bangles{s})$-coordinate upon applying $i_s^*\circ\Forg_{\bangles{s}}$ is of the form $-\alpha(\varphi_{B_j}(s))$ as in \eqref{diag lambda transport}, where $\Bb_j$ is a connected component of $\Bb^s$ as in the proof of Lemma \ref{BernsteinAction}. We conclude again that the square commutes. Finally, note that $\Spec(\gamma)^*[\Ll_\sigma]=\Ll_\sigma$, as the Schubert varieties and their conormal bundles are defined over $\Q$.
This holds as equivariant sheaves: concretely, the equivariant structure is a vector bundle morphism $G\times\relSpec{\Sym\Ll_\sigma^\vee}\to\relSpec{\Sym\Ll_\sigma^\vee}$ over $T^*_{\bar{Y}_\sigma}(\Bb\times\Bb)$ commuting with $G$ action and fibrewise linear, and this morphism is defined over $\Q$. Therefore this square commutes.
\item
Now we check the second square, where, \textit{e.g.}, the top horizontal map sends $\alpha\otimes (z\otimes f\otimes\Ff)\mapsto \alpha\gamma(z)(f\circ\gamma)\otimes\Spec(\gamma)^*\Ff$. It clearly commutes and all the maps in the diagram are isomorphisms.
\item
We explain why the third square commutes. Recall from \cite[\S 5.11]{CG} that
\[
\lambda_s:=\bigotimes_{\alpha\in\mathrm{sp}(\Nn)}\sum_{i}(-\alpha(s))^i\Lambda^iN_\alpha,
\]
where $N=\bigoplus_{\alpha\in\mathrm{sp}(N)}N_\alpha$ is the $\bangles{s}$-weight decomposition of the normal bundle $N$ to $\Nn^s$, and that 
$\Oo_{\Nn}\boxtimes\lambda_s^{-1}$ exists in complexified $K$-theory \cite[Cor. 5.11.3]{CG}. It is then clear that $\id\cdot\gamma\otimes\Spec(\gamma)^*\lambda_s=\lambda_{\gamma(s)}$, and hence the same is true for the inverse classes. Here, we again use that the considerations of \eqref{diag lambda transport} apply also the bundles $\Ll_\sigma$, by their construction recalled above.
Therefore this square commutes.
\item
For the fourth square, note that the bivariant Riemann-Roch map 
\[
\mathrm{RR}\colon K_0(\St^{(s,q)})_\Q\overset{\sim}{\to} A_*(\St^{(s,q)})_\Q
\]
given by $\mathrm{RR}(\Ff)=(1\boxtimes\mathrm{Td}(\Nn)\cap\Ch_*(\Ff)$ of \cite[Thm. 5.11.11]{CG} differs from the ``usual" Grothendieck-Riemann-Roch isomorphism $\tau$ of \cite[Cor. 18.3.2]{Fulton} (By Theorem 18.3 (3) of \textit{loc. cit.}, $\tau$ agrees with the definition of Theorem 18.2 of \textit{loc. cit.}) by the invertible Todd class factor $\mathrm{Td}(\Nn)\boxtimes 1$, which is fixed under $\Spec(\gamma)^*$---this follows for instance from \cite[Thm.18.2 (3)]{Fulton} and fact that $\Td(\Nn)=\tau(\Oo_{\Nn})$ (see p.354 of \textit{loc. cit.}). Therefore \cite[Thm.18.2 (3)]{Fulton} applied again to the lci morphism $\Spec(\gamma)$ shows that 
\[
\mathrm{RR}(\Spec(\gamma)^*\Ff)=\Spec(\gamma)^*\mathrm{RR}(\Ff)
\]
for $\Ff\in K_0(\St^{(s,q)})_\Q$; as $\Spec(\gamma)$ is an isomorphism of quasi-projective schemes, its relative Todd class (see B.7.6 of \textit{op. cit.}) is trivial. Therefore the fourth square commutes.
\item
Finally we check the fifth square. Let $\xi$ be an algebraic cycle in $\St^s$, and $v$ a cycle in $\Bb_N^s$. The right-down circuit gives, for $1\otimes 1\otimes\xi$ and $z_1\otimes z_2\otimes v$,
\begin{equation}
\label{act pre}
z_1\gamma(z_2)\otimes \Spec(\gamma)_{\Bb_N^s}^* p_{1*}(\xi\cap p_2^*v),
\end{equation}
where $p_i$ are the respective projections $\St^s\to\Nn^s$ and we view $v$ as a cycle in the smooth, by \cite[Lem. 5.11.1]{CG}, variety $\Nn^s$ (to which Theorem 6.2.4 of \textit{op. cit.} also applies). Now we apply base-change \cite[Prop. 1.7]{Fulton} (see \cite[\S 2.7]{CG} for the corresponding statements for homology throughout) to the diagram
\[
\begin{tikzcd}
\St^{\gamma(s)}\arrow[d, "p_1"]\arrow[r, "\Spec(\gamma)_{\St^s}"]&\St^s\arrow[d, "p_1"]
\\
\Nn^{\gamma(s)}\arrow[r, "\Spec(\gamma)"]&\Nn^s
\end{tikzcd}
\]
noting that $p_1$ is proper and the isomorphism $\Spec(\gamma)$ of schemes is certainly flat. Therefore \eqref{act pre} becomes
\begin{align*}
z_1\gamma(z_2)\otimes p_{1*}\left(\Spec(\gamma)^*_{\St^s}\left(\xi\cap p_{2*}v\right)\right)&=
z_1\gamma(z_2)\otimes p_{1*}\left(\Spec(\gamma)_{\St^s}^*\xi\cap \Spec(\gamma)_{\Nn^s\times\Nn^s}^*p_2^*v\right)
\\
&=
z_1\gamma(z_2)\otimes p_{1*}\left(\Spec(\gamma)_{\St^s}^*\xi\cap p_2^*\Spec_{\Nn^s}(\gamma)^*v\right)
\end{align*}
Here, the first equality used that the intersection pairing commutes, by \cite[Thm. 6.2 (b)]{Fulton}, with pullback by the vertical flat maps in the diagram 
\[
\begin{tikzcd}
\gamma(\xi)
\arrow[r, hook]
\arrow[d, "\Spec(\gamma)_{\St^s}"]
&
\Nn^{\gamma(s)}\times\Nn^{\gamma(s)}
\arrow[d, "\Spec(\gamma)_{\Nn^s\times\Nn^s}"]\\
\xi\arrow[r, hook]&\Nn^s\times\Nn^s.
\end{tikzcd}
\]
The second equality follows from commutativity of the obvious square, viewing $v$ as a cycle in the ambient smooth variety $\Nn^s$.
\end{enumerate}
\end{proof}
The proof of Theorems \ref{thm standard} and \ref{thm Clozel} for inner forms of $\GL_n$ is now complete. 
\qed
\begin{rem}
In general, the action of $\Ll_\sigma$ on $A_*(\Bb_N^s)_\C$ depends only on the conjugacy class of Cartan subgroup of (the Levi quotient of) the disconnected group $Z_G(N)$; see \cite[3.4 (a)]{CellsIV}, or \cite[Lem. 14, \S3.5]{rigid} for this observation in the language of the operator Paley-Wiener theorem for Hecke algebras. In either perspective, either because $Z_{\GL_n}(N)$ is connected, or because the only $I$-spherical discrete series of each Levi subgroup of $\GL_n$ is the tensor product of Steinberg representations, respectively, it follows that the action depends only on $N$ in this case.
\end{rem}

\bibliography{standard_modules_biblio.bib}

\end{document}

%% file: preamble.tex
\newcommand{\diag}{\mathrm{diag}}
\newcommand{\pair}[2]{\left\langle #1 , #2\right\rangle}

\DeclareMathOperator{\Sym}{Sym}

\newcommand{\sets}[2]{\left\{#1\,\middle|\,#2\right\}}

\newcommand{\bangles}[1]{\left\langle #1 \right\rangle}
\newcommand{\sett}[1]{\left\{#1\right\}}

\newcommand{\stab}[2]{\mathrm{Stab}_{#1}(#2)}
\newcommand{\rquotient}{\backslash}

\DeclareMathOperator{\End}{End}

\newcommand{\id}{\mathrm{id}}

\DeclareMathOperator{\Ind}{Ind}

\newcommand{\Rep}{\mathbf{Rep}}
\newcommand{\Mod}{\mathbf{Mod}}

\newcommand{\Coh}{\mathrm{Coh}}

\DeclareMathOperator{\Spec}{Spec}
\DeclareMathOperator{\relSpec}{\underline{Spec}}

\newcommand{\Pp}{\mathbb{P}}

\newcommand{\Bb}{\mathcal{B}}

\newcommand{\pt}{\mathrm{pt}}

\newcommand{\Ff}{\mathcal{F}}

\newcommand{\Ee}{\mathcal{E}}
\newcommand{\Ll}{\mathcal{L}}

\newcommand{\dInt}[2]{\int {#1} \,\mathrm{d} {#2}}

\newcommand{\Z}{\mathbb{Z}}
\newcommand{\C}{\mathbb{C}}
\newcommand{\Q}{\mathbb{Q}}

\newcommand{\Oo}{\mathcal{O}}

\newcommand{\Ad}{\mathrm{Ad}}

\newcommand{\g}{\mathfrak{g}}

\DeclareMathOperator{\Lie}{\mathrm{Lie}}
\newcommand{\Nn}{\tilde{\mathcal{N}}}

\newcommand{\sS}{\mathfrak{s}}

\newcommand{\GL}{\mathrm{GL}}

\newcommand{\SL}{\mathrm{SL}}

\newcommand{\Gm}{{\mathbb{G}_\mathrm{m}}}

\newcommand{\Aut}{\mathrm{Aut}}

\newcommand{\St}{\mathrm{St}}
\newcommand{\triv}{\mathrm{triv}}

\newcommand{\bq}{\mathbf{q}}

\newtheorem{theorem}{Theorem}
\newtheorem*{theorem*}{Theorem}
\newtheorem{cor}{Corollary}
\newtheorem{prop}{Proposition}
\newtheorem{lem}{Lemma}

\theoremstyle{definition}
\newtheorem{dfn}{Definition}
\newtheorem*{dfn*}{Definition}
\theoremstyle{remark}
\newtheorem{ex}{Example}
\newtheorem*{ex*}{Example}
\newtheorem{rem}{Remark}